\documentclass{elsarticle}

\usepackage{amssymb,amsthm,amsmath}

\theoremstyle{definition}
\newtheorem{thm}{Theorem}[section]
\newtheorem{Cor}[thm]{Corollary}
\newtheorem{prop}[thm]{Proposition}
\newtheorem{lemma}[thm]{Lemma}
\newtheorem{ex}[thm]{Example}
\newtheorem*{remark}{Remark}

\newcommand{\lr}[1]{\langle #1 \rangle}

\journal{Journal of Algebra}

\begin{document}

\begin{frontmatter}

\title{Modular Centralizer Algebras Corresponding to $p$-Groups}
\author{Adam Allan}
\address{Department of Mathematics, University of Chicago, 5734 S. University Avenue, Chicago, IL 60637, USA.}

\begin{abstract}
We study the Loewy structure of the centralizer algebra $kP^Q$ for $P$ a $p$-group with subgroup $Q$ and $k$ a field of characteristic $p$. Here $kP^Q$ is a special type of Hecke algebra. The main tool we employ is the decomposition $kP^Q = kC_P(Q) \ltimes I$ of $kP^Q$ as a split extension of a nilpotent ideal $I$ by the group algebra $kC_P(Q)$. We compute the Loewy structure for several classes of groups, investigate the symmetry of the Loewy series, and give upper and lower bounds on the Loewy length of $kP^Q$. Several of these results were discovered through the use of MAGMA, especially the general pattern for most of our computations. As a final application of the decomposition, we determine the representation type of $kP^Q$.
\end{abstract}

\begin{keyword}
centralizer algebra \sep Hecke algebra \sep Loewy series \sep radical series \sep nilpotency index
\end{keyword}

\end{frontmatter}

\section{Preliminaries}

If $G$ is a finite group with subgroup $H$ and $k$ a commutative ring with identity, then as in ~\cite{CurtisReiner}, the centralizer algebra $kG^H$ consists of all elements of $kG$ that are invariant under the conjugation action of $H$. There have been several recent investigations into the representation theory of $kG^H$ in the papers ~\cite{Ellers94}, ~\cite{Ellers95}, ~\cite{Ellers00}, ~\cite{Ellers04}, ~\cite{Murray02}, and ~\cite{Murray04}. In these papers, one of the motivating problems is the identification of the block idempotents of $kG^H$ for $G$ a $p$-solvable group and $H \unlhd G$, or $G = S_n$ and $H = S_m$. For $P$ a $p$-group with subgroup $Q$ and $k$ a field of characteristic $p$, $kP^Q$ has no nontrivial idempotents, and therefore the questions one might ask concerning the structure and representation theory of $kP^Q$ have a somewhat different flavor than the study of the more general $kG^H$. In particular, this paper explores the Loewy structure of $kP^Q$ and its representation type.\\

Jennings proved in ~\cite{Jennings} a theorem that now bears his name and which allows us to the compute the radical layers of the group algebra $kP$ for $P$ a $p$-group using certain characteristic subgroups $\{\kappa_i\}$ of $P$. More precisely, we let $\kappa_1 = P$ and inductively define $\kappa_n$ as the subgroup of $P$ generated by $[\kappa_s,\kappa_t]$ whenever $s,t < n$ and $s+t \ge n$, along with all $p$th powers of elements from $\kappa_r$ whenever $r < n$ and $pr \ge n$. So $\kappa_2 = \Phi(P)$ and each $\kappa_i/\kappa_{i+1}$ is an elementary abelian $p$-group. Let $\{x_{ij}\}_{j=1}^{s_i}$ be a collection of elements of $\kappa_i$ whose images in $\kappa_i/\kappa_{i+1}$ form a basis. Define $X_{ij} = x_{ij}-1 \in kP$ and consider products of the form $\prod X_{ij}^{a_{ij}}$ where $0 \leq a_{ij} < p$, and where the identity occurs when all $a_{ij} = 0$. We define the weight $w$ of $\prod X_{ij}^{a_{ij}}$ as $\sum ia_{ij}$. Jennings' Theorem states that the products of weight $w$ lie in $J^w(kP)$ and form a basis modulo $J^{w+1}(kP)$.\\

Alperin generalized in ~\cite{AlperinLoewy} the preceding discussion to compute the radical layers of the $kP$-module $k[P/R]$ for $R \leq P$. In particular, for each $i$ the set $\{y_{ij}\}$ is chosen as a subset of $\{x_{ij}\}$ such that the image of $\{y_{ij}\}$ in $\kappa_i/(\kappa_i \cap R)\kappa_{i+1}$ forms a basis. Again we let $Y_{ij} = y_{ij}-1 \in k[P/R]$, consider products of the form $\prod Y_{ij}^{a_{ij}}$ for $0 \leq a_{ij} < p$, and assign this product the weight $w = \sum ia_{ij}$. Then the products of weight $w$ lie in $J^w(k[P/R])$ and form a basis modulo $J^{w+1}(k[P/R])$.\\

More recently, in ~\cite{Towers} Towers obtained a decomposition of the Hecke algebra $\mbox{End}_P(k_Q\uparrow^P)$ for $P$ a $p$-group with subgroup $Q$, and used this to compute the radical series of $\mbox{End}_P(k_Q\uparrow^P)$ when $P$ has nilpotency class 2 and $[P,Q]$ is cyclic. Taking $Q = 1$ yields $\mbox{End}_P(k_Q\uparrow^P) \simeq kP$, and hence one can view these results as a generalization of the work from \cite{AlperinLoewy}. This is similar to what occurs for $kP^Q$ when $Q \leq ZP$, in which case $kP^Q = kP$. Since $kP^Q$ is isomorphic with the Hecke algebra $\mbox{End}_{Q \times P}(k_{\Delta Q}\uparrow^{Q \times P})$, several of the results from ~\cite{Towers} are relevant in the study of $kP^Q$. Indeed, the decomposition obtained in ~\cite{Towers} for Hecke algebras of $p$-groups specializes to the decomposition $kP^Q = kC_P(Q) \ltimes I$, and Theorem 6.2 from ~\cite{Towers} will serve as the starting point for our computations of the radical series of $kP^Q$.\\

The paper proceeds as follows. In section 2 we briefly derive the decomposition $kP^Q = kC_P(Q) \ltimes I$ using a counting principle that will foreshadow later arguments. We apply this decomposition to obtain several general results regarding the structure of $kP^Q$. In section 3 we establish several technical results that are of interest in their own right and that will also be useful later. In section 4 we compute the dimensions of the radical layers for centralizer algebras arising from extra special $p$-groups, noting that there is a surprising symmetry present here. The extra special $p$-groups also arise in sections 8 and 9. This symmetry is further explored in section 5, under somewhat restrictive conditions on $kP^Q$. In section 6 we give some general results concerning the Loewy length $\ell\ell(kP^Q)$ of $kP^Q$. For $P$ a $p$-group with a cyclic subgroup of index $p$ and $Q \leq P$ arbitrary, $\ell\ell(kP^Q)$ is computed explicitly in section 7. In sections 8 and 9 we derive lower and upper bounds for $\ell\ell(kP^Q)$. As an application of the decomposition $kP^Q = kC_P(Q) \ltimes I$, we determine the representation type of $kP^Q$ in section 10. Lastly, section 11 poses open questions and possible avenues for further research.\\

\textit{Notation} Throughout this paper $P$ will denote a $p$-group with subgroup $Q$, and $k$ will be an arbitrary field of characteristic $p$. Only in section 10 will it be necessary to assume that $k$ is algebraically closed. The augmentation map is denoted by $\varepsilon : kP \rightarrow k$. For $x,y \in P$ we write ${^xy} = xyx^{-1}$, $[x,y] = xyx^{-1}y^{-1}$, and if $y = {^qx}$ for some $q \in Q$ then we denote this by $x \sim_Q y$. In particular, the orbit ${^Qx}$ of $x$ under the conjugation action of $Q$ consists of all $y \in P$ such that $y \sim_Q x$. For $\xi \in kP$ if we write $\xi = \sum_{x \in P} \xi_x x$, then we define the support of $\xi$ as $\mbox{Supp}(\xi) = \{ x \in P : \xi_x \not= 0 \}$. For $x \in P$ we let $\kappa_x$ be the element of $kP^Q$ given by $\kappa_x = \sum_{q \in Q/C_Q(x)} {^qx}$. Lastly, $\ell\ell(kP^Q)$ denotes the Loewy (= radical) length of $kP^Q$. That is, $\ell\ell(kP^Q)$ is the smallest integer $d$ for which $J^d(kP^Q) = 0$.

\section{Structure of $kP^Q$}

With the notation from section 1, notice that $J(kP^Q) = \mbox{Ker}(\varepsilon) \cap kP^Q$ since $\mbox{Ker}(\varepsilon) \cap kP^Q$ is a nilpotent ideal in $kP^Q$ with codimension 1. In particular, $kP^Q$ is a basic and connected algebra with the unique simple module $k$, on which it acts via $\varepsilon$. The decomposition $kP^Q = kC_P(Q) \ltimes I$ can be obtained from \cite{Towers} by appropriate translation using the isomorphism $kP^Q \simeq \mbox{End}_{Q \times P}(k_{\Delta Q}\uparrow^{Q \times P})$. However, it is useful and instructive to derive this result directly. Begin by letting $\{x_i\}$ be representatives of the orbits of $Q$ acting on $P-C_P(Q)$ by conjugation, and define $\Omega = \{ \kappa_{x_i} \}$. So $kP^Q$ has the basis $C_P(Q) \cup \Omega$, and it is clear that $kC_P(Q)$ is a subalgebra of $kP^Q$. Let $I$ be the $k$-linear span of $\Omega$; we claim that $I$ is an ideal. Since $c\kappa_x = \kappa_{cx}$ and $\kappa_x c = \kappa_{xc}$ for $c \in C_P(Q)$ and $\kappa_x \in \Omega$, it is clear that $cI = Ic = I$. Also, if $\kappa_x,\kappa_y \in \Omega$ and $c \in C_P(Q)$, then the coefficient of $c$ in $\kappa_x\kappa_y$ equals $|S|$ where

$$S = \{(\bar{q_1},\bar{q_2}) \in Q/C_Q(x) \times Q/C_Q(y) : q_1xq_1^{-1}q_2yq_2^{-1} = c\}$$

The diagonal action of $Q$ on $Q/C_Q(x) \times Q/C_Q(y)$ induced by left multiplication leaves $S$ invariant since $qcq^{-1} = c$ for $q \in Q$. Also, if $(\bar{q_1},\bar{q_2})$ is invariant under $Q$ then $Q \leq q_1C_Q(x)q_1^{-1} \cap q_2C_Q(y)q_2^{-1}$, and so we obtain the contradiction that $x,y \in C_P(Q)$. Therefore, $Q$ acts semiregularly on $S$, and thus $|S| = 0$ in $k$. This implies that $I$ is an ideal in $kP^Q$. Moreover, since $p \mid |\mbox{Supp}(\kappa_x)|$ for $\kappa_x \in \Omega$, we see $I \subset \mbox{Ker}(\varepsilon)$. So $I$ is a nilpotent ideal, and hence $J(kP^Q) = J(kC_P(Q)) \oplus I$.\\

Moreover, $I$ is a $kC_P(Q) = kC$ permutation module since $C$ acts on the basis $\Omega$ by both left and right multiplication. When $C$ acts on $\Omega$ on the left and $\kappa_x \in \Omega$, we let $S_x$ denote the stabilizer in $C$ of $\kappa_x$. That is, $S_x$ consists of all $c \in C$ such that $cx \sim_Q x$. Writing $cx = qxq^{-1}$ yields $c = [q,x]$. So if write $S_x^{\dagger}$ for the set of all commutators $[q,x]$ with $q \in Q$, then $S_x = C \cap S_x^{\dagger}$. The following proposition summarizes these results.\\

\begin{prop}If $P$ is a $p$-group with subgroup $Q$, then as in ~\cite{Towers} there is a decomposition $kP^Q = kC_P(Q) \ltimes I$ with $I$ a nilpotent ideal that has a basis $\Omega$ that is permuted via left (or right) multiplication by $C_P(Q)$. Moreover, $J(kP^Q) = J(kC_P(Q)) \oplus I$ and the stabilizer $S_x$ in $C_P(Q)$ of $\kappa_x \in \Omega$ satisfies $S_x = C_P(Q) \cap S_x^{\dagger}$.\end{prop}

Recall that if $G$ is a finite group then $kG$ is a symmetric algebra. Using the decomposition from Proposition 2.1, we show that if $Q$ is a non-central subgroup of $P$, then $kP^Q$ is never a symmetric algebra.\\

\begin{prop}If $P$ is a $p$-group with a non-central subgroup $Q$, then $kP^Q$ is not self-injective.\end{prop}
\begin{proof}For convenience write $\Lambda = kP^Q$. If $_{\Lambda}\Lambda$ were injective, then $(_{\Lambda}\Lambda)^*$ would be projective and hence isomorphic to $\Lambda_{\Lambda}$ since $\Lambda$ is local, so that $\mbox{Top}((_{\Lambda}\Lambda)^*) \simeq k$ and hence $\mbox{Soc}(_{\Lambda}\Lambda) \simeq k$. Therefore, it suffices to show that $\mbox{Soc}(_{\Lambda}\Lambda)$ is at least two-dimensional. Write $\sigma = \sum_{p \in P} p$ and notice that $\sigma \in kP^Q$ and $\xi\sigma = \varepsilon(\xi)\sigma$ for $\xi \in kP^Q$. Thus $J(kP^Q)\sigma = 0$ and hence $k\sigma \subseteq \mbox{Soc}(_{\Lambda}\Lambda)$. On the other hand, $I$ is a nonzero submodule of $_{\Lambda}\Lambda$, and so $0 \not= \mbox{Soc}(I) \subseteq \mbox{Soc}(_{\Lambda}\Lambda)$. Since $\mbox{Supp}(\sigma) = P$ and $\mbox{Supp}(\xi) \subseteq P \setminus C_P(Q)$ for $\xi \in I$ we get $\mbox{Soc}(I) \cap k\sigma = 0$, and hence $\mbox{Soc}(_{\Lambda}\Lambda)$ is at least two-dimensional.
\end{proof}

Recall also that if $H$ is a subgroup of $G$ then $kG$ is projective as a $kH$-module. Again using proposition 2.1 we can show that the analogous statement for centralizer algebras is false.

\begin{prop}If $P$ is a $p$-group with a non-central $p$-element $x$, then $kP$ is not projective as a $kP^{\lr{x}}$-module.\end{prop}
\begin{proof}
Write $kP^{\lr{x}} = kC_P(x) \ltimes I$ where $I$ has basis $\Omega$. Since $|x| = p$,  we get $|\mbox{Supp}(\kappa_y)| = p$ for $\kappa_y \in \Omega$. So we compute

$$\dim kP^{\lr{x}} = |C_P(x)| + |\Omega| = |C_P(x)| + \frac{|P|-|C_P(x)|}{p} > \frac{|P|}{p}$$

If $kP$ were projective as a $kP^{\lr{x}}$-module, then $kP$ would be a free $kP^{\lr{x}}$-module since $kP^{\lr{x}}$ is a local algebra. This would imply that $\dim kP^{\lr{x}}$ divides $|P|$. Since $\dim kP^{\lr{x}} > \frac{|P|}{p}$, the only is possibility is $\dim kP^{\lr{x}} = |P|$, and hence $x \in ZP$, contrary to our assumption on $x$.
\end{proof}

\section{Radical Structure of $kP^Q$}

With the notation from section 2, we know that $J(kP^Q) = J' \oplus I$ where we write $C = C_P(Q)$ and $J' = J(kC)$ for brevity. In computing $J^d$ for $d > 1$ it is useful to consider two separate questions: when is $J'I = IJ'$? and when is $I^2 \subseteq J'I$?

\begin{prop}Let $P$ be a $p$-group with subgroup $Q$ and write $kP^Q = kC \ltimes I$. Then $J'I = IJ'$ if and only if $Q$ satisfies the following condition: for all $x \in P$ and all $c \in C$, there exists $q \in Q$ such that $[x,qc] \in C$.
\end{prop}

\begin{proof}
To establish $IJ' \subseteq J'I$ it is enough to check that $\kappa_x(c^{-1}-1) \in J'I$ for $\kappa_x \in \Omega$ and $c \in C$. We compute

$$\kappa_x(c^{-1}-1) = (c^{-1}-1)\kappa_{cxc^{-1}} + (\kappa_{cxc^{-1}}-\kappa_x)$$

Since $(c^{-1}-1)\kappa_{cxc^{-1}} \in J'I$ we need $\kappa_{cxc^{-1}}-\kappa_x \in J'I$. So let $\Omega = \coprod_{i=1}^e \Omega_j$ be an orbit decomposition of $\Omega$ under the left action of $C$, so that $I = \bigoplus k\Omega_j$ as $kC$-modules. Then $J'I = \bigoplus J'k\Omega_j$ where $J'k\Omega_j$ is the Jacobson radical of the $kC$-module $k\Omega_j$. If $\kappa_{x_j}$ is a representative of the orbit $\Omega_j$, then $J'k\Omega_j$ consists of all elements of the form $\sum_{z \in C / S_{x_j}} \lambda_z z\kappa_{x_j}$ with $\sum \lambda_z = 0$. Assuming that $\kappa_{cxc^{-1}} \in \Omega_{j_1}$ and $\kappa_x \in \Omega_{j_2}$, we see that $\kappa_{cxc^{-1}}-\kappa_x \in J'I$ if and only if $j_1 = j_2$. In other words, there must exist $z \in C$ and $q \in Q$ satisfying $cxc^{-1} = q^{-1}zxq$. Since this is equivalent to $z = qcxc^{-1}q^{-1}x^{-1} = [qc,x]$, we see $IJ' \subseteq J'I$ if and only if for all $x \in P-C$ and $c \in C$ there is $q \in Q$ such that $[x,qc] \in C$. This condition can be taken for all $x \in P$ since it is trivially satisfied for $x \in C$. This is the same condition that is necessary and sufficient for $J'I \subseteq IJ'$. So the proposition is established.
\end{proof}

\begin{Cor}If $Q \leq P$ then $J'I = IJ'$ whenever $P$ has nilpotency class 2, $C_P(Q) \unlhd P$, or $Q$ contains its centralizer.\end{Cor}
\begin{proof}If $P$ has class 2 then $[x,qc] \in ZP \leq C$; if $C \unlhd P$ we can take $q = 1$; and if $C \leq Q$ we can take $q = c^{-1}$.
\end{proof}

For convenience, we refer to condition (*) as the requirement that $J'I = IJ'$, and we refer to condition (**) as the requirement that $I^2 \subseteq J'I$. Condition (*) appears to be mild as it is almost always satisfied. For example, of the 28,075 pairs $(P,Q)$ with $|P| = 2^6$, MAGMA computed that there are only 68 for which $J'I \not= IJ'$. Condition (**) on the other hand seems less natural; of the 28,075 pairs $(P,Q)$ with $|P| = 2^6$, there are 7,347 for which $I^2 \not\subseteq J'I$. Condition (**) is also more obscure in how it is reflected in the group-theoretic structure of $P$ and $Q$, and we are only able to offer a partial result for when (**) holds. For this, we need to generalize a result from ~\cite{Towers} concerning the structure constants of $I$ in terms of the basis $\Omega$.\\

\begin{lemma}Suppose $Q \leq P$ and $\kappa_x,\kappa_y,\kappa_z \in \Omega$. Then $\kappa_z$ appears in $\kappa_x\kappa_y$ with nonzero coefficient $\mu_{xyz}$ only when $z \sim_Q q^{-1}xqy$ for some $q \in Q$, in which case $\mu_{xyz} = |q^{-1}S_{x^{-1}}^{\dagger}q \cap S_y^{\dagger}|$. Moreover, if $[P,Q,Q] = 1$ then $\mu_{xyz} = |[Q,x^{-1}] \cap [Q,y]|$, $S_w = [Q,w]$ for $\kappa_w \in \Omega$, and $[Q,P]$ is abelian.\end{lemma}
\begin{proof}
There is an anti-isomorphism $\psi : \mbox{End}_{Q \times P}(k_{\Delta Q}\uparrow^{Q \times P}) \rightarrow kP^Q$ of $k$-algebras given by sending the endomorphism $f$ to $\xi$ where $\xi \in kP^Q$ is the unique element satisfying $f(\Delta Q) = (1,\xi)\Delta Q$. Let $\{A_{(q,p)}\}$ denote the standard basis of $\mbox{End}_{Q \times P}(k_{\Delta Q}\uparrow^{Q \times P})$ indexed by the double cosets of $\Delta Q$ in $Q \times P$. It is easy to check that $\psi(A_{(q,p)}) = \kappa_{q^{-1}p}$. For $x,y \in P$, if the basis element $\kappa_z$ appears in $\kappa_x\kappa_y$ with a nonzero coefficient, then $z = q_1xq_1^{-1}q_2yq_2^{-1}$ for some $q_1,q_2 \in Q$, and hence $z \sim_Q q^{-1}xqy$ where $q = q_1^{-1}q_2$. Moreover, the coefficient of $\kappa_z$ in $\kappa_x\kappa_y$ equals the coefficient of $A_{(q,xqy)}$ in $A_{(1,y)}A_{(1,x)}$, which by ~\cite{Towers} satisfies

\begin{equation}\label{coefficients}
\mu_{xyz} = |Q|^{-1}|\Delta Q (1,x^{-1}) \Delta Q (q,xqy) \cap \Delta Q (1,y) \Delta Q|
\end{equation}

It is understood that the division in (\ref{coefficients}) takes place in $\mathbb{Z}$ and actually does yield a rational integer. We rewrite (\ref{coefficients}) to obtain

\begin{equation}\nonumber
\begin{split}
\mu_{xyz} &= |Q|^{-1}|\Delta Q (1,x^{-1}) \Delta Q (1,x) (q,q) \cap \Delta Q (1,y) \Delta Q (1,y^{-1})| \\
&= |Q|^{-1}|\Delta Q (1 \times S_{x^{-1}}^{\dagger})(q,q) \cap \Delta Q (1 \times S_y^{\dagger})| \\
&= |Q|^{-1}|\Delta Q (1 \times q^{-1}S_{x^{-1}}^{\dagger}q) \cap \Delta Q (1 \times S_y^{\dagger})| \\
&= |Q|^{-1}|\Delta Q \times (q^{-1}S_{x^{-1}}^{\dagger}q \cap S_y^{\dagger})| \\
&= |q^{-1}S_{x^{-1}}^{\dagger}q \cap S_y^{\dagger}| \\
\end{split}
\end{equation}\nonumber

where the fourth equality occurs because $(q_1,q_1)(1,p_1) = (q_2,q_2)(1,p_2)$ precisely when $q_1 = q_2$ and $p_1 = p_2$. This establishes the first assertion. When $[P,Q,Q] = 1$ we see as in ~\cite{IsaacsGroups} that $[P,Q] = [Q,P]$ is abelian. If $q_1,q_2 \in Q$ and $w \in P$ then the fact that $Q$ centralizes $[Q,P]$ and $[Q,P]$ is abelian yields

\begin{equation}\label{PQQ}
\begin{split}
[q_1q_2,w] &= [q_1,w][q_2,w] \\
[q_1,w]^{-1} &= [q_1^{-1},w] \\
[q_2q_1q_2^{-1},w] &= [q_1,q_2wq_2^{-1}] = [q_1,w] \\
[q_1,w^{-1}] &= w^{-1}[q_1^{-1},w]w \\
\end{split}
\end{equation}

The first of these relations implies that $[Q,w] = S_w^{\dagger}$ for $\kappa_w \in \Omega$. Since $q^{-1}[Q,x]q = [Q,x]$, we obtain $\mu_{xyz} = |[Q,x^{-1}] \cap [Q,y]|$ and the proof is complete.
\end{proof}

\begin{remark} Unfortunately, the condition $[P,Q,Q] = 1$ is rather restrictive since it implies that $[Q,Q,P] = 1$ and hence $[Q,Q] \leq ZP \cap Q \leq ZQ$, so that $Q$ has nilpotency class at most 2.\end{remark}

\begin{thm}Assume $P$ has subgroup $Q$ satisfying $[P,Q,Q] =1$, and write $kP^Q = kC \ltimes I$. If $I^2 \not\subseteq J'I$ then there are $x,y \in P-C_P(Q)$ satisfying $Q = C_Q(x)C_Q(y)$ and $C_Q(xy) = C_Q(x) \cap C_Q(y)$. Moreover, $|C_P(Q)| = |Q:C_Q(xy)|$.\end{thm}
\begin{proof}
By Lemma 3.3, if $\kappa_x,\kappa_y,\kappa_z \in \Omega$ then $\kappa_z$ occurs in $\kappa_x\kappa_y$ with nonzero coefficient $\mu_{xyz}$ only when $z \sim_Q qxq^{-1}y$ for some $q \in Q$, in which case $\mu_{xyz} = |[Q,x^{-1}] \cap [Q,y]|$. Since $\mu_{xyz}$ is independent of $z$, we see that $\kappa_x\kappa_y = 0 \in J'I$ provided $1 < [Q,x^{-1}] \cap [Q,y]$. So assume $[Q,x^{-1}] \cap [Q,y] = 1$ and $z = qxq^{-1}y$. Using the fact that $S_z = [Q,z]$, write $\Omega_z = \{ c\kappa_z : c \in C/[Q,z] \}$ and $\Omega = \Omega_z \coprod \Omega_z'$ for $\Omega_z' = \Omega - \Omega_z$, so that $I = k\Omega_z \oplus k\Omega_z'$ as a $kC$-module. Since $\kappa_z$ was chosen arbitrarily, $\kappa_x\kappa_y \in J'I$ if and only if $\pi_{\Omega_z}(\kappa_x\kappa_y) \in J'k\Omega_z$. If we define

$$T_z = \{ c[Q,z] \in C/[Q,z] : \text{$cz \sim_Q q_1xq_1^{-1}y$ for some $q_1 \in Q$} \}$$

then $\pi_{\Omega_z}(\kappa_x\kappa_y)$ is a sum of $|T_z|$ many linearly independent elements in $\Omega_z$, each appearing with coefficient 1. Notice that $J'k\Omega_z$ consists of all elements of the form $\sum_{c \in C/[Q,z]} \lambda_c c\kappa_z$ with $\sum \lambda_c = 0$. Therefore, $\pi_{\Omega_z}(\kappa_x\kappa_y) \in J'k\Omega_z$ precisely when $|T_z| \equiv_p 0$. It is hence necessary to obtain a criterion for membership in $T_z$. Suppose $c[Q,z] \in T_z$ and write $cz = q_2(q_1xq_1^{-1}y)q_2^{-1}$. Then $cz = c[q,x]xy = q_2[q_1,x]xyq_2^{-1}$ and using (\ref{PQQ}) we rewrite this as

$$c = [q_1,x][q_2,xy][q,x]^{-1} = [q_1q^{-1},x][q_2,xy] \in [Q,x][Q,xy]$$

Reversing the argument, we see that $c[Q,z] \in T_z$ whenever $c \in [Q,x][Q,xy]$. Thus, $T_z$ consists of all $c[Q,z]$ with $c \in [Q,x][Q,xy]$, and since $[Q,x][Q,xy] \leq C$ a standard counting argument yields

$$|T_z| = \frac{|[Q,x][Q,xy]|}{|[Q,x][Q,xy] \cap [Q,z]|}$$

Thus, $p \mid |T_z|$ precisely when $[Q,x][Q,xy] \not\leq [Q,z]$. Using $(\ref{PQQ})$ and $z = qxq^{-1}y$, we obtain for $q_1 \in Q$ the equality

$$[q_1,qxq^{-1}y] = [q_1,xq^{-1}yq] = [q_1,x]{^x[q_1,q^{-1}yq]} = [q_1,x]{^x[q_1,y]} = [q_1,xy]$$

and so $[Q,z] = [Q,xy]$. So $p \mid |T_z|$ precisely when $[Q,x] \not\leq [Q,xy]$. To summarize, since $\kappa_x,\kappa_y \in \Omega$ were chosen arbitrarily, we see that $I^2 \not\subseteq J'I$ precisely when there exist $x,y \in P-C_P(Q)$ for which $[Q,x^{-1}] \cap [Q,y] = 1$ and $[Q,x] \leq [Q,xy]$.\\

Suppose $x,y \in P-C_P(Q)$ satisfy $[Q,x^{-1}] \cap [Q,y] = 1$ and $[Q,x] \leq [Q,xy]$. So for every $q_1 \in Q$ there is $q_2 \in Q$ for which $[q_1,x] = [q_2,xy] = [q_2,x]{^x[q_2,y]}$. By (\ref{PQQ}) we see $[q_1^{-1}q_2,x^{-1}] = [q_2,y]$ is an element of $[Q,x^{-1}] \cap [Q,y] = 1$. That is, $q_2 \in C_Q(y)$ and $q_1^{-1}q_2 \in C_Q(x)$, so that $q_1 \in C_Q(y)C_Q(x)$ and hence $Q = C_Q(y)C_Q(x)$. For $q \in Q$ write $q = q_1q_2$ with $q_2 \in C_Q(x)$ and $q_1 \in C_Q(y)$. Then $qxq^{-1}y = q_1xyq_1^{-1} \sim_Q xy$. This implies that $\kappa_x\kappa_y = \mu\kappa_{xy}$ for some $\mu \in k$. In fact, this argument shows that the equality $\kappa_x\kappa_y = \mu\kappa_{xy}$ holds when we work over $\mathbb{Z}$ instead of $k$, where we now have $\mu \in \mathbb{N}$. This has the advantage of demonstrating $|\mbox{Supp}(\kappa_x)||\mbox{Supp}(\kappa_y)| = \mu|\mbox{Supp}(\kappa_{xy})|$. For $w \in P$ we know $|\mbox{Supp}(\kappa_w)| = |Q:C_Q(w)|$. Since $C_Q(x) \cap C_Q(y) \leq C_Q(xy)$, we compute $\mu = |C_Q(xy):C_Q(x) \cap C_Q(y)|$. Since $\kappa_x\kappa_y \not\in J'I$, we know $\kappa_x\kappa_y \not= 0$ and thus $C_Q(xy) = C_Q(x) \cap C_Q(y)$.\\

Moreover, for every $c \in C$ we get $c\kappa_{xy} = \kappa_{xy}$, and hence $cxy = q'xyq'^{-1}$ for some $q' \in Q$. This yields $c = [q',xy]$ and hence $C_P(Q) \leq [Q,xy]$. Of course $C_P(Q) = [Q,xy]$ since $Q$ centralizes $[P,Q]$. By (\ref{PQQ}) the map $Q \times [Q,xy] \rightarrow [Q,xy]$ given by $(q,\xi) \mapsto [q,xy]\xi$ yields a transitive action of $Q$ on $[Q,xy]$ with point stabilizer $C_Q(xy)$. We conclude that $|C_P(Q)| = |Q:C_Q(xy)|$, and so the theorem is established.\\
\end{proof}

Our proof in Theorem 3.4 actually allows us to establish a strong converse.

\begin{Cor}Suppose $Q \leq P$ and there are $x,y \in P-C_P(Q)$ for which $Q = C_Q(x)C_Q(y)$ and $C_Q(xy) = C_Q(x) \cap C_Q(y)$. Then $I^2 \not\subseteq J'I$ where $kP^Q = kC \ltimes I$ and $J' = J(kC)$.\end{Cor}

\begin{remark}It is particularly easy to compute $J^d$ when (*) and (**) hold. For instance, since $J = J' \oplus I$ we compute $J^2 = J'^2 \oplus (J'I + IJ' + I^2) = J'^2 \oplus J'I$, and more generally $J^d = J'^d \oplus J'^{d-1}I$ for all $d \geq 1$. This is the same conclusion as reached in ~\cite{Towers} for $P$ of class 2 and $[P,Q]$ cyclic.\end{remark}

We will also need in section 5 the following fact.

\begin{lemma}Suppose $Q \leq P$ and $kP^Q = kC \ltimes I$ where $I$ has basis $\Omega$ and $\Omega = \coprod_{i=1}^e \Omega_i$. If $p$ is odd then $e \geq 2$.\end{lemma}
\begin{proof}
Clearly $e \geq 1$ since $Q$ is non-central. Suppose for the sake of contradiction that $e = 1$. In other words, $C$ acts transitively on $\Omega$ so that $|\Omega| = |C:S_x|$ for any fixed $\kappa_x \in \Omega$. Notice that $|\mbox{Supp}(\kappa_y)| = |\mbox{Supp}(\kappa_x)|$ for all $y \in P-C$, and so $|\Omega| = (|P|-|C|)/|\mbox{Supp}(\kappa_x)|$. This yields $|S_x|(|P:C|-1) = |\mbox{Supp}(\kappa_x)|$. But $|S_x|, |P:C|,$ and $|\mbox{Supp}(\kappa_x)|$ are all powers of $p$, and $|P:C|-1 > 1$ since $Q$ is non-central and $p$ is odd, and so we have a contradiction.
\end{proof}

\section{Extra Special $p$-Groups}

We compute the dimensions of the successive radical layers $J^d(kP^Q)/J^{d+1}(kP^Q)$ of $kP^Q$ for $P$ an extra special $p$-group and $Q$ an arbitrary subgroup of $P$. These numbers have the surprising property that they are symmetric in $d$ when $p = 2$ or $p = 3$ and $Q$ is well-chosen. This symmetry holds for group algebras of arbitrary $p$-groups by Jennings' Theorem, but does not hold for centralizer algebras of $p$-groups in general. The Loewy series of $kP^Q$ will be encapsulated in terms of a Poincar\'{e} polynomial. Recall as in ~\cite{BensonI} that if $P$ is a $p$-group with subgroups $\{ \kappa_i \}$ then the Loewy series for $kP$ has the Poincar\'{e} polynomial

\begin{equation}\label{q}
q(t) = \prod_{n = 1}^{\infty} (1 + t^n + t^{2n} + \cdots + t^{(p-1)n})^{\dim \kappa_n/\kappa_{n+1}}
\end{equation}

Similarly, the Loewy series of $k[P/R]$ has the Poincar\'{e} polynomial

\begin{equation}\label{r}
r(t) = \prod_{n = 1}^{\infty} (1 + t^n + t^{2n} + \cdots + t^{(p-1)n})^{\dim \kappa_n/(\kappa_n \cap R)\kappa_{n+1}}
\end{equation}

\begin{thm}Suppose $P$ is an extra special $p$-group and $Q \leq P$. If $C = C_P(Q)$ is not elementary abelian then $kP^Q$ has the Loewy series given by

$$p(t) = (1 + t + \cdots + t^{p-1})^{\log_p |C|-1}((|P:C|-1)t + 1 + t^2 + t^4 + t^6 + \cdots + t^{2(p-1)})$$

If $C = C_P(Q)$ is elementary abelian then $kP^Q$ has the Loewy series given by

$$p(t) = (1 + t + \cdots + t^{p-1})^{\log_p |C|-1}((|P:C|-1)t + 1 + t + t^2 + t^3 + \cdots + t^{p-1})$$
\end{thm}

\begin{proof}

As usual, write $kP^Q = kC \ltimes I$ where $I$ has basis $\Omega$. If $x \in P-C$ and $q \in Q$, then ${^qx} = [q,x]x \in P'x$, and so ${^Qx} \subseteq P'x$. Since $|{^Qx}|$ is a power of $p$ larger than 1 and $|P'| = p$, we see that $|{^Qx}| = p$, and hence $|\mbox{Supp}(\kappa_x)| = p$ for $\kappa_x \in \Omega$. Therefore $|\Omega| = (|P|-|C|)/p$. Now write $\Omega = \Omega_1 \cup \cdots \Omega_e$ as a union of orbits under the action of $C$ on $\Omega$. If $\kappa_{x_i}$ is a representative of $\Omega_i$ then the stabilizer in $C$ of $\kappa_{x_i}$ equals $C \cap [Q,x_i] = ZP$ since $P$ has class 2. Moreover, $\Omega_i \simeq k[C/ZP]$ and so $\Omega_i \simeq \Omega_1$ as $kC$-modules for all $i$. Since $e|C:ZP| = |\Omega|$, we obtain $e = |P:C|-1$.\\

Since $P$ has class 2, we know that $J'I = IJ'$ by Corollary 3.2. Furthermore, $I^2 = 0 \subseteq J'I$ by Lemma 3.3. By the remarks following Corollary 3.5, we conclude that $J^d(kP^Q) = J'^d \oplus J'^{d-1}I$. In particular, if the Loewy series for $kC$ has the Poincar\'{e} polynomial $q(t)$ and the Loewy series for $k\Omega_1$ has the Poincar\'{e} polynomial $r(t)$, then the Loewy series for $kP^Q$ has the Poincar\'{e} polynomial $p(t) = q(t)+(|P:C|-1)tr(t)$. So it suffices to compute $q(t)$ and $r(t)$. This breaks down into two cases: $C$ not elementary abelian or $C$ elementary abelian.\\

Suppose $C$ is not elementary abelian. Then $\kappa_2(C) = \Phi(C) = ZP$ and $\kappa_i(C) = 1$ for $i > 2$. Moreover, $\dim \kappa_1/\kappa_2 = \dim \kappa_1/(\kappa_1 \cap ZP)\kappa_2 = \log_p |C| - 1$, $\dim \kappa_2/\kappa_3 = 1$, and $\dim \kappa_2/(\kappa_2 \cap ZP)\kappa_3 = 0$. From (\ref{q}) and (\ref{r}) we get

$$q(t) = (1 + t + \cdots + t^{p-1})^{\log_p |C|-1}(1+t^2+\cdots+t^{2(p-1)})$$

$$r(t) = (1+t+\cdots+t^{p-1})^{\log_p |C|-1}$$

and thus $p(t)$ is as claimed. If $C$ is elementary abelian, then the situation is even simpler: $\kappa_1(C) = C$ and $\kappa_i(C) = 1$ for $i > 1$, and hence $\dim \kappa_1/\kappa_2 = \log_p |C|$ and $\dim \kappa_1/(\kappa_1 \cap ZP)\kappa_2 = \log_p |C| - 1$. So (\ref{q}) and (\ref{r}) yield

$$q(t) = (1 + t + \cdots + t^{p-1})^{\log_p |C|}$$

$$r(t) = (1+t+\cdots+t^{p-1})^{\log_p |C|-1}$$

and again $p(t)$ is as claimed.
\end{proof}

In particular, we obtain the following symmetry.

\begin{Cor}If $P$ is an extra special $p$-group and $Q \leq P$ is non-central, then $kP^Q$ has a symmetric Loewy series in precisely the following cases.

\begin{enumerate}

\item $p=2$ and $C_P(Q)$ is not elementary abelian.

\item $p=3$ and $C_P(Q)$ is elementary abelian.

\end{enumerate}

In particular, if $y \in P-ZP$ then $kP^{\lr{y}}$ has a symmetric Loewy series when: $p = 2$ and $|P| \geq 32$; $|P| = 8$ and $|y| = 4$; $|P| = 27$ and $P$ has exponent 3; or $|P| = 27$, $P$ has exponent 9, and $|y| = 3$.
\end{Cor}
\begin{proof}

First observe that if $n \in \mathbb{N}$ and $s(t)$ is a polynomial for which $(1+t+\cdots+t^n)s(t)$ is symmetric in $t$, then $s(t)$ is also symmetric in $t$. This is established by induction on $\deg s(t)$. Conversely, if $s(t)$ is symmetric in $t$ then so is $(1+t+\cdots+t^n)s(t)$. So let $p(t)$ be the Poincar\'{e} polynomial for the Loewy series $kP^Q$. Then $p(t)$ is symmetric in $t$ for $C_P(Q)$ not elementary abelian iff

$$1 + (|P:C_P(Q)|-1)t + t^2 + t^4 + t^6 + \cdots + t^{2(p-1)}$$

is symmetric in $t$. Since $|P:C_P(Q)|-1 \geq 1$, this occurs precisely when $p = 2$. Moreover, $p(t)$ is symmetric in $t$ for $C_P(Q)$ elementary abelian iff

$$1 + (|P:C_P(Q)|-1)t + t^2 + t^3 + \cdots + t^{p-1}$$

is symmetric in $t$. This occurs precisely when $p = 3$. This establishes the first assertion. Now assume that $p = 2$, $y \in P-ZP$, and $|P| \geq 32$. We need to show that $C_P(y)$ is not elementary abelian. In Theorem 4.1 we proved $|{^Py}| = 2 = |P:C_P(y)|$. So assume $C_P(y)$ is abelian, so that $P = \lr{x}C_P(y)$ and $C_P(x) \cap C_P(y) \leq ZP$ for $x \in P-C_P(y)$. Since $|P:C_P(x)| = 2$ and $P = C_P(x)C_P(y)$ we get

$$|Z(P)| \geq |C_P(x) \cap C_P(y)| = \frac{|C_P(x)||C_P(y)|}{|C_P(x)C_P(y)|} = \frac{|P|}{4} \geq 8$$

This contradiction shows that $C_P(y)$ is non-abelian, thus establishing the symmetry of the Loewy series of $kP^{\lr{y}}$. The three small cases follow similarly.
\end{proof}

\begin{remark}If $P$ is a $p$-group then Jennings' Theorem yields the corollary that the socle series and radical series of $kP$ coincide in reverse order. Unfortunately, Proposition 2.2 precludes this possibility.
\end{remark}

\section{Symmetry of Loewy Structure}

We provide a partial answer to the question raised in section 4 of precisely when $kP^Q$ has a symmetric Loewy series. More precisely, assuming that $kP^Q$ satisfies conditions (*) and (**), we provide necessary and sufficient conditions for the Loewy series of $kP^Q$ to be symmetric in terms of the group-theoretic structure of $P$ and $Q$. Most interestingly, we see that symmetry arises only when $p = 2,3$.

\begin{prop}Assume $P$ is a $p$-group with non-central subgroup $Q$, $C = C_P(Q)$, and $kP^Q$ satisfies conditions (*) and (**). Then $kP^Q$ has a symmetric Loewy series if and only if one of the following conditions holds:

\begin{enumerate}

\item[(a)] $p = 2$ and whenever $x \in P-C$ either (i) $S_x$ is an elementary abelian subgroup of $C$ of order 4 that intersects $\Phi(C)$ trivially or (ii) $S_x$ is a subgroup of $C$ of order 2 contained in $\Phi(C)$ that intersects $[C,\Phi(C)]\mho^1(\Phi(C))$ trivially.\\

\item[(b)] $p = 2$, there is $x^* \in P-C$ such that $S_{x^*} = 1$, and $S_x$ is a subgroup of $C$ of order 2 that intersects $\Phi(C)$ trivially whenever $x \in P-C$ with $\kappa_x \not= \kappa_{x^*}$.\\

\item[(c)] $p = 3$ and $S_x$ is a subgroup of $C$ of order 3 that intersects $\Phi(C)$ trivially whenever $x \in P - C$.

\end{enumerate}

\end{prop}

\begin{proof}

Assume that $kP^Q$ has a symmetric Loewy series, and let $\Omega_i = \coprod_1^e \Omega_i$ be an orbit decomposition of $\Omega$ under the left action of $C$. Also let $S_i$ denote the stabilizer in $C$ of some $\kappa_{x_i} \in \Omega_i$. Since (*) and (**) hold we get $J^d(kP^Q) = J'^d \oplus \bigoplus J'^{d-1}k\Omega_i$ where $J' = J(kC)$. So if $q(t)$ and $r_i(t)$ are the Poincar\'{e} polynomials for the radical series of $kP^Q$ and $k\Omega_i$, respectively, then the radical series for $kP^Q$ has the Poincar\'{e} polynomial $p(t) = q(t)+t\sum_{i=1}^e r_i(t)$. By Jennings' Theorem, $q$ and $r_i$ are symmetric in $t$, and of course $\deg r_i \leq \deg q$. Observe that the constant and leading coefficients of $q$ and $r_i$ equal 1. For $1 \leq j \leq \deg q$ let $e_j$ denote the number of indices $i$ for which $\deg r_i = \deg q - j$. There are two cases to consider: $e_0 = 0$ and $e_0 > 0$.\\

Assume that $e_0 = 0$, so that $p(t) = 1 + (q'(0)+e)t + \cdots + (1+e_1)t^{\deg q}$. Since $p$ is symmetric in $t$ we conclude that $e_1 = 0$, and hence $t^{\deg q-1}$ appears in $p(t)$ with coefficient $q'(0)+e_2$. This implies that $e_2 = e$, and hence $\deg r_i = \deg q - 2$ for all $i$. From (\ref{r}) see that $p = 2,3$. If $p = 3$ then

$$\dim \kappa_n/(\kappa_n \cap S_i)\kappa_{n+1} = \begin{cases}
\dim \kappa_1/S_i\kappa_2-1 &\text{if $n = 1$}\\
\dim \kappa_n/\kappa_{n+1} &\text{if $n > 1$}
\end{cases}$$

In particular $\kappa_n \cap S_i \leq \kappa_{n+1}$ for $n \geq 2$, and so $\kappa_2 \cap S_i \leq \kappa_n \cap S_i = 1$ for $n$ sufficiently large. Moreover $|S_i| = |S_i\kappa_2 / \kappa_2| = 3$. Since $\kappa_2 = \Phi(C)$, this establishes the necessity of the condition given for $p = 3$. On the other hand, if $p = 2$ then

$$\dim \kappa_n/(\kappa_n \cap S_i)\kappa_{n+1} = \begin{cases}
\dim \kappa_1/S_i\kappa_2-2 &\text{if $n = 1$}\\
\dim \kappa_n/\kappa_{n+1} &\text{if $n > 1$}
\end{cases}$$

or else

$$\dim \kappa_n/(\kappa_n \cap S_i)\kappa_{n+1} = \begin{cases}
\dim \kappa_n/\kappa_{n+1} &\text{if $n \not= 2$}\\
\dim \kappa_2/\kappa_3-1 &\text{if $n = 2$}
\end{cases}$$

In the first case we get $S_i \cap \Phi(C) = 1$ and $S_i$ elementary abelian of order 4; in the second case we get $S_i \leq \Phi(C)$, $S_i$ intersects $\kappa_3(C) = [C,\Phi(C)]\mho^1(\Phi(C))$ trivially, and $|S_i| = 2$.\\

Assume now that $e_0 > 1$. Since the leading coefficient of $p$ equals $e_0$, we obtain $e_0 = 1$. That is, there is a unique $i^*$ for which $\deg r_{i^*} = \deg q$. Hence $p(t) = 1 + (q'(0)+e)t + \cdots + (1+q'(0)+e_1)t^{\deg q} + t^{\deg q+1}$ and so $e_1 = e-1$. By Lemma 3.6, either $p = 2$ or $e_1 > 0$. If $e_1 > 0$ then $\deg r_i = \deg q -1$ for some $i$, and so we conclude from (\ref{r}) that $p = 2$. Therefore $p = 2$, and arguing as we did above, we see that $S_{i^*} = 1$, and $S_i$ intersects $\Phi(C)$ trivially for $i \not= i^*$ with $|S_i| = 2$.\\

We have established the necessity of the conditions under the assumption of symmetry, and in effect, we have proven that they are sufficient to guarantee symmetry, as well.

\end{proof}

\section{Loewy Length; General Considerations}

Recall that we denote by $\ell\ell(kP^Q)$ the Loewy length of $kP^Q$, which is the smallest $d \geq 1$ for which $J^d(kP^Q) = 0$. For instance, if $P$ is an extra special $p$-group, then in the statement of Theorem 4.1, $\ell\ell(kP^Q) = \deg p(t) + 1$. In particular, for extra special $p$-groups, $\ell\ell(kP^Q)$ is determined by $C_P(Q)$. This behavior will also be apparent in the next section for another distinguished class of $p$-groups. As a prelude to some of those computations, and to give an idea of how this behavior might hold for more general groups, we offer the following proposition.

\begin{thm}If $P$ has abelian subgroup $Q$ and $C_P(Q) \unlhd P$, then $\ell\ell(kP^Q) = \ell\ell(kC_P(Q))$.\end{thm}
\begin{proof}We may assume $Q \not\leq ZP$. Since $C \lhd P$ we have $[C,P] \leq C$ and hence $[C,P,Q] = [Q,C,P] = 1$. By Hall's Three Subgroups Lemma we obtain $[P,Q,C] = 1$ and in particular $[P,Q,Q] = 1$. Writing $kP^Q = kC \ltimes I$ and $J' = J(kC)$, we claim that $I^2 \subseteq J'I$. If this is not true, then by Theorem 3.4 there is $p \in P$ for which $|C| = |Q:C_Q(p)|$. Since $Q$ is abelian, $|Q| \leq |C|$ and thus $Q = C$ and $C_Q(p) = 1$. But $1 < C_Q(p)$ since $1 < Q \lhd P$. This contradiction establishes $I^2 \subseteq J'I$. From this and Corollary 3.2, we know $J^d(kP^Q) = J'^d \oplus \bigoplus_1^e J'k\Omega_i$ where $\Omega = \coprod_1^e \Omega_i$ is an orbit decomposition of $\Omega$. Therefore, $\ell\ell(kP^Q) = \max\{\ell\ell(kC),\ell\ell(k\Omega_i)+1\}$. Since $1 < S_x^{\dagger} \leq C$ for all $x \in P-C$, we see that $1 < S_x$, and hence $\ell\ell(k\Omega_i) < \ell\ell(kC)$ for all $i$ by (\ref{q}) and (\ref{r}). So we obtain $\ell\ell(kP^Q) = \ell\ell(kC_P(Q))$.
\end{proof}

\begin{ex}If $P$ is a metacyclic $p$-group with cyclic subgroup $R \unlhd P$ such that $P/R$ cyclic, then $\ell\ell(kP^Q) = \ell\ell(kC_P(Q))$ whenever $Q \leq R$ by Theorem 6.1. In particular, $C_P(Q)$ is metacyclic since $R \unlhd C_P(Q)$ and $C_P(Q)/R$ is cyclic, and hence $\ell\ell(kC_P(Q))$ may be computed as in ~\cite{Koshitani} and ~\cite{Motose2}.\end{ex}

In the proof to Proposition 6.1 we utilized the following useful lemma.

\begin{lemma}Suppose $P$ is a $p$-group with subgroup $R \not= 1$. Then the $kP$-module $k[P/R]$ satisfies $\ell\ell(k[P/R]) < \ell\ell(kP)$.\end{lemma}

One might ask what can be said in the general case, where $\ell\ell(kP^Q)$ need not equal $\ell\ell(kC_P(Q))$. As with any finite dimensional algebra, $\ell\ell(kP^Q) \leq \dim kP^Q = |C|+|\Omega|$. However, this estimate can be improved, as in the following.

\begin{prop}If $P$ has subgroup $Q$ then either $\ell\ell(kP^Q) = \ell\ell(kC_P(Q))$ or $\ell\ell(kC_P(Q)) < \ell\ell(kP^Q) \leq |\Omega|+1$.
\end{prop}
\begin{proof}
Write $J = J(kP^Q)$, $J' = J(kC)$, and inductively define ideals $L_i$ of $kP^Q$ contained in $I$ by $L_1 = I$ and $L_{i+1} = JL_i + L_iJ$. Since $J^i = (J'+I)^i$ we obtain $J^i = J'^i \oplus L_i$ for all $i \geq 1$. If $L_i = L_{i+1}$ for some $i \geq 1$, then Nakayama's Lemma applied on the left yields $L_i = L_iJ$, and Nakayama's lemma applied on the right yields $L_i = 0$. So if $L_i = L_{i+1}$ for some $1 \leq i \leq \ell\ell(kC)-1$, then $J^{\ell\ell(kC)} = 0$ and hence $\ell\ell(kP^Q) = \ell\ell(kC)$. Assuming $\ell\ell(kP^Q) \not= \ell\ell(kC)$ we obtain $L_{i+1} < L_i$ for all $1 \leq i \leq \ell\ell(kC)$, and in particular $\dim kP^Q/J^{\ell\ell(kC)} \geq |C| + \ell\ell(kC)-1$. Since $\dim kP^Q = |C|+|\Omega|$ and $\dim J^{\ell\ell(kC)}/J^{\ell\ell(kP^Q)} \geq \ell\ell(kP^Q)-\ell\ell(kC)$, we obtain $\ell\ell(kP^Q) \leq |\Omega|+1$, as required.
\end{proof}

The following example shows that $C_P(Q)$ does not in general determine $\ell\ell(kP^Q)$.

\begin{ex}Let $Q = D_{16}$ be the dihedral group with generators $a$ and $b$ of order 8 and 2, respectively. Write $\psi$ for the automorphism of $Q$ of order 2 given by $\psi(a) = a^3$ and $\psi(b) = b$, and let $P = Q \rtimes \lr{\psi}$. Since $\psi$ is an outer automorphism, $C_P(Q) = ZQ = \lr{a^4}$ and so also $C_P(P) = ZP = \lr{a^4}$ since $|a^4| = 2$. In general, if $N \leq N_P(Q)$, then $kN^Q$ is a subalgebra of $kP^Q$ with $J(kN^Q) = kN^Q \cap J(kP^Q)$, and hence $\ell\ell(kN^Q) \leq \ell\ell(kP^Q)$. This applies in particular with $N = Q$ where $\ell\ell(kQ^Q) = 4$ by Theorem 7.3. On the other hand, working with the canonical basis for $ZkP = kP^P$ one can compute $J^2(ZkP) = 0$, and hence $\ell\ell(ZkP) = 2$.
\end{ex}

\section{$p$-groups with Cyclic Subgroup of index $p$}

Let $P$ be a $p$-group of order $p^n$ that contains a cyclic subgroup of index $p$. Our aim is to compute $\ell\ell(kP^Q)$ for all $Q \leq P$. For convenience define $l = p^{n-2}$. If $P$ is cyclic, then $\ell\ell(kP^Q) = \ell\ell(kP) = p^n$ by (\ref{q}). If $P$ is abelian and non-cyclic, then $\ell\ell(kP^Q) = \ell\ell(kP) = p^{n-1}+p-1$ by ~\cite{Koshitani}. Assume then that $P$ is nonabelian, so that $n \geq 3$, and using the notation from ~\cite{Aschbacher} we get $P = \text{Mod}_{p^n}, D_{2^n}, SD_{2^n}$ for $n \geq 4$, or $Q_{2^n}$. In the first three cases, $P$ has a presentation of the form $\lr{a,b | a^{p^{n-1}} = b^p = 1,{^ba} = a^i}$ where $i = l+1,-1$ or $l-1$ respectively, and in the fourth case $P$ has the presentation $\lr{a,b | a^{p^{n-1}} = 1, b^p = a^l,{^ba} = a^{-1}}$.\\

Let's begin with the simplest case: $P = \text{Mod}_{p^n}$. With the presentation given above we see that $ZP = \lr{a^p}$, $[P,P] = \lr{a^l}$ has order $p$, $[P,P] \leq ZP$, and hence $P$ has nilpotency class 2. As in Theorem 4.1, since $[Q,P]$ is cyclic for $Q \leq P$, we obtain $J^d(kP^Q) = J'^d \oplus \bigoplus_{i=1}^e J'^{d-1}k\Omega_i$ where $J' = J(kC)$ and $\Omega = \coprod_{i=1}^e \Omega_i$ is an orbit decomposition. Therefore $\ell\ell(kP^Q) = \max\{ \ell\ell(kC_P(Q)), \ell\ell(k\Omega_i)+1 \}$. Since $1 < C_P(Q) \cap [Q,x]$ for all $\kappa_x \in \Omega$, we know $\ell\ell(k\Omega_i) < \ell\ell(kC_P(Q))$ by Lemma 6.3. Therefore $\ell\ell(kP^Q) = \ell\ell(kC_P(Q))$, and so it suffices to analyze $C_P(Q)$. Notice that either $Q \leq ZP$, $Q = \lr{a}$, $c \in Q$ for some $c \in P-\lr{a}$ and $Q \cap \lr{a} \leq ZP$, or $Q = P$. Then $C_P(Q)$ is given respectively as $C_P(Q) = P,\lr{a},\lr{a^p,c}$, or $\lr{a^p}$. Accordingly, since $C_P(Q)$ is abelian with a cyclic subgroup of index $p$, we obtain the following.

\begin{thm}Let $P = \text{Mod}_{p^n}$ with elements $a$ and $b$ of orders $p^{n-1}$ and $p$, respectively, where ${^ba} = a^{l+1}$ and $l = p^{n-2}$. If $Q \leq P$ then

$$\ell\ell(kP^Q) = \begin{cases}
p^{n-1}+p-1 &\text{if $C_P(Q) = P$}\\
p^{n-1} &\text{if $C_P(Q) = \lr{a}$}\\
p^{n-2}+p-1 &\text{if $C_P(Q) = \lr{a^p,c}$ for some $c \in P-\lr{a}$}\\
p^{n-2} &\text{if $C_P(Q) = \lr{a^p}$}
\end{cases}$$

\end{thm}

Alternatively, all cases in Theorem 7.1 except $Q = P$ are subsumed under Theorem 6.1. Now assume that $P = D_{2^n}$, $SP_{2^n}$ with $n \geq 4$, or $Q_{2^n}$, with its respective presentation as given above. Also write $z = a^l$ with $l = 2^{n-2}$ so that $ZP = \lr{z}$ and $|z| = 2$. Choose $s = 0,1$ so that ${^cq} = q^{sl-1}$ for all $c \in P-\lr{a}$ and $q \in \lr{a}$. If $Q \leq P$ then either $Q \leq ZP$, $ZP < Q \leq \lr{a}$, $c \in Q$ for some $c \in P-\lr{a}$ and $Q \cap \lr{a} \leq ZP$, or $c \in Q$ for some $c \in P-\lr{a}$ and $ZP < Q \cap \lr{a}$. Since $C_P(c) = \lr{c,z}$ for all $c \in P-\lr{a}$ and $C_P(a^i) = \lr{a}$ for all $a^i \in \lr{a}-ZP$, we see that $C_P(Q)$ is given respectively as $C_P(Q) = P, \lr{a}, \lr{c,z}$, or $\lr{z}$. As Theorem 7.3 will show, $\ell\ell(kP^Q)$ is determined by $C_P(Q)$. Before proving this theorem, it will be helpful to analyze the subalgebra $\Lambda_a$ of $k\lr{a}$ fixed under the conjugation action of any $c \in P-\lr{a}$.\\

\begin{lemma}Assume $\lr{a}$ is a cyclic group of order $2^{n-1} \geq 4$ and $s = 0,1$. Further, if $s = 1$ then assume $2^{n-1} \geq 8$. Define $l = 2^{n-2}$, $z = a^l$, and $\eta_i = a^i+a^{-i}z^{si} \in k\lr{a}$ for $i \in \mathbb{Z}$. Also let $k$ be a field of characteristic 2. If $\psi$ is the $k$-algebra isomorphism induced by $\psi(a) = a^{sl-1}$ and $\Lambda_a$ is the subalgebra fixed under $\psi$, then $\Lambda_a$ is spanned as a vector space by $\Delta = \{1,z,\eta_1,\eta_2,\ldots,\eta_{2l-1}\}$. Moreover, $J^l(\Lambda_a) = 0$ and $J^{l-1}(\Lambda_a)$ contains $\zeta = a+a^3+a^5+\cdots+a^{2l-1}$.\end{lemma}

\begin{proof}
Since $\psi$ permutes the basis $\{a^i : 0 \leq i < 2^{n-1}\}$ of $k\lr{a}$, it follows that $\Lambda_a$ consists of all $\sum \lambda_i a^i$ with $\lambda_{\cdot}$ constant on the orbits of $\psi$. In particular, $\Delta$ spans $\Lambda_a$ as a vector space. Since $J(\Lambda_a) = \Lambda_a \cap \ker(\varepsilon : k\lr{a} \rightarrow k)$ we see that $J(\Lambda_a)$ is spanned by $\{ 1+z,\eta_1,\ldots,\eta_{2l-1} \}$. So $J^{l-1}(\Lambda_a)$ contains $\eta_1^{l-1} = a+a^3+a^5+\cdots+a^{2l-1}$, where $\eta_1^{l-1}$ is computed by using ~\cite{Fine}. Since $l \geq 2$, $(1+z)^2 = 0$, and $\Lambda_a$ is commutative, we see that the elements of $J^l(\Lambda)$ are $k$-linear combinations of elements of the form $\eta\theta$ where $\theta \in J(\Lambda_a)$ and $\eta$ is a product of some $l-1$ many elements in $\{\eta_i\}_{i=1}^{2l-1}$. We claim that for all $1 \leq i \leq 2l-1$ there is $\theta_i \in \Lambda_a$ for which $\eta_i = \eta_1 \theta_i$. This is obvious for $\eta_1$ and $\eta_2$ since $\eta_2 = \eta_1^2$. So suppose $i > 2$ and the result is true for $i' < i$, and write $\eta_{i-2} = \eta_1\theta_{i-2}$. Then $\eta_i = \eta_{i-1}\eta_1 + \eta_{i-2}z^s = \eta_1(\eta_{i-1}+\theta_{i-2}z^s)$, thus establishing the result by induction. Since $\eta_1^l = z+z = 0$ and $(1+z)\eta_1^{l-1} = 0$, we see that $J(\Lambda_a)$ annihilates $\eta_1^{l-1}$. Therefore $J(\Lambda_a)$ annihilates $J^{l-1}(\Lambda_a)$, so that $J^l(\Lambda_a) = 0$, thus completing the proof.
\end{proof}

\begin{thm}Let $P = D_{2^n}$, $SD_{2^n}$ with $n \geq 4$, or $Q_{2^n}$. If $Q \leq P$ then

$$\ell\ell(kP^Q) = \begin{cases}
2^{n-1}+1 &\text{if $C_P(Q) = P$}\\
2^{n-1} &\text{if $C_P(Q) = \lr{a}$}\\
2^{n-2}+1 &\text{if $C_P(Q) = \lr{z,c}$ for some $c \in P-C_P(Q)$}\\
2^{n-2} &\text{if $C_P(Q) = \lr{z}$}\\
\end{cases}$$

\end{thm}

\begin{proof}
If $C_P(Q) = P$ then $kP^Q = kP$ and hence $\ell\ell(kP^Q) = 2^{n-1}+1$ by ~\cite{Koshitani}. If $C_P(Q) = \lr{a}$ then $ZP < Q \leq \lr{a}$, and hence $\ell\ell(kP^Q) = \ell\ell(kC_P(Q))$ by Theorem 6.1, and $\ell\ell(kC_P(Q)) = 2^{n-1}$ since $C_P(Q)$ is cyclic.\\

Assume $C_P(Q) = \lr{c,z}$ for some $c \in P-\lr{a}$, so that $\lr{c} \leq Q \leq \lr{c,z}$. Then $kP^Q$ is spanned as a vector space by $\{1,c,z,cz,\eta_1,\ldots,\eta_{2l-1},\eta_1 c,\ldots,\eta_{2l-1}c\}$ in analogy with Lemma 7.2, where $\eta_i = a^i + a^{-i}z^{si}$. In particular, $kP^Q$ is commutative and $J(kP^Q)$ is spanned by $1+c,1+z,1+cz$ and all $\eta_i$ and $\eta_i c$. Define $\Lambda_c$ as the subalgebra of $kP^Q$ generated by 1 and $c$, and define $\Lambda_a$ as in Lemma 7.2. Then by inspection

$$J(kP^Q) = J(\Lambda_a)\Lambda_c + \Lambda_a J(\Lambda_c)$$

For example, $1+cz = (1+z)c + (1+c)$. Since $J^2(\Lambda_c) = 0$ and $\Lambda_a$ commutes element-wise with $\Lambda_c$, we obtain for all $d \geq 2$

$$J^d(kP^Q) = J^d(\Lambda_a)\Lambda_c + J^{d-1}(\Lambda_a) J(\Lambda_c)$$

Lemma 7.2 yields $J^{l+1}(kP^Q) = 0$ and $J^l(kP^Q) \not= 0$ since it contains the nonzero element $\zeta(1+c)$. Therefore $\ell\ell(kP^Q) = 2^{n-2}+1$.\\

Lastly, assume that $C_P(Q) = \lr{z}$, so that $c \in Q$ for some $c \in P-\lr{a}$ and $ZP < Q \cap \lr{a}$. Then $kP^Q$ is spanned by $\{1,z,\eta_1,\ldots,\eta_{2l-1}\} \cup \{ \kappa_{q'c} | q' \in \lr{a} \}$ and $J(kP^Q)$ is spanned by $1+z$ and all $\eta_i$ and $\kappa_{q'c}$. If $Q \cap \lr{a} = \lr{q}$ then define $\tau = \sum_{g \in \mho^1(Q)} g$, and observe that $p \mid \mbox{Supp}(\tau) \not= 0$ since $ZP < Q \cap \lr{a}$. Since ${^q(q'c)} = q'q^2q^{sl}c$ for $q' \in Q$, we conclude that $\kappa_{q'c} = q'q^{sl}\tau c = q'\tau c$. Notice that $cq'\tau c^{-1} = q'\tau$, so that $q'\tau \in \Lambda_a$, and hence $q'\tau \in J(\Lambda_a)$. So by inspection

$$J(\Lambda_a) \subseteq J(kP^Q) \subseteq J(\Lambda_a)\Lambda_c$$

Induction on $d \geq 1$ yields

$$J^d(\Lambda_a) \subseteq J^d(kP^Q) \subseteq J^d(\Lambda_a)\Lambda_c$$

Lemma 7.2 implies that $J^l(kP^Q) = 0$ and $0 \not= \zeta \in J^{l-1}(kP^Q)$. Therefore $\ell\ell(kP^Q) = 2^{n-2}$, thus completing the theorem.
\end{proof}

\begin{remark}Notice that if $P$ is as in Theorem 7.3 with $n \geq 4$, then $2 = \ell\ell(kZP) < \ell\ell(ZkP)$. By Proposition 6.4 we get $2^{n-2} \leq |\Omega|+1$. In fact, $P$ has $2^{n-2}+3$ many conjugacy classes so that $|\Omega| = 2^{n-2}+1$ and hence $\ell\ell(ZkP) = |\Omega|-1$. There are also a handful of cases for $n = 4$ where $Q < P$ and $\ell\ell(kP^Q) = |\Omega|-1$.
\end{remark}

As an important corollary to Theorems 7.1 and 7.3 we obtain the following.

\begin{Cor}Suppose $P$ is a noncyclic $p$-group with a cyclic subgroup of index $p$, and that $Q \leq P$. Then $\ell\ell(kP^Q) = p^{n-1}+p-1$ if and only if $Q \leq ZP$. Moreover, $\ell\ell(kP^Q) = p^{n-1}$ if and only if $C_P(Q)$ is a cyclic subgroup of index $p$ in $P$.\end{Cor}

\begin{remark}If $\psi$ is an automorphism of the $p$-group $P$ with $\psi(Q) = R$ for $Q,R \leq P$, then $kP^Q \simeq kP^R$ as $k$-algebras. Using this observation, one can show that for $P = \text{Mod}_{p^n}$ there are only 4 centralizer algebras $kP^Q$ that arise, up to $k$-algebra isomorphism, as $Q$ ranges across the subgroups of $P$. This behavior does not hold for $D_{2^n},SD_{2^n}$, or $Q_{2^n}$, since one may show that the number of centralizer algebras grows linearly in $n$. For example, if $P = D_{2^n}$ then we may take $Q_i = \lr{a^{2^i}}$ to obtain a centralizer algebra of dimension $\frac{1}{2}|P|+2^{i+1}$. It is in light of this that Theorem 7.3 is somewhat surprising.
\end{remark}

\section{Lowest Bounds on Loewy Length}

If $P$ is a group of order $p^n$ then it is known ~\cite{Wallace} that $\ell\ell(kP) \geq n(p-1)+1$, and in particular, $\ell\ell(kP) \geq p$. This lower bound also holds for centralizer algebras.

\begin{lemma}If $P$ is a $p$-group with subgroup $Q$ then $\ell\ell(kP^Q) \geq p$.
\end{lemma}
\begin{proof}
Write $kP^Q = kC \ltimes I$ so that $J(kP^Q) = J(kC) \oplus I$. Since $J^d(kC) \subseteq J^d(kP^Q)$ for all $d \geq 1$, we see that $\ell\ell(kP^Q) \geq \ell\ell(kC) \geq p$.
\end{proof}

The next proposition characterizes the centralizer algebras $kP^Q$ with minimal possible Loewy length; those for which $\ell\ell(kP^Q) = p$.

\begin{prop}If $Q \leq P$ then $\ell\ell(kP^Q) = p$ if and only if we have: $C_P(Q) = ZP$ is generated by an element $z$ of order $p$, and $zx \sim_Q x$ for all $x \in P-ZP$.
\end{prop}

\begin{proof}
Write $kP^Q = kC \ltimes I$ and $J' = J(kC)$. Assuming $\ell\ell(kP^Q) = p$, we see by the proof to Lemma 8.1 that $|C| = p$, and hence $ZP = C$ is generated by an element $z$ of order $p$. Moreover, since $J'^{p-1} \not= 0$ and $J'^{p-1}I = 0$ we conclude that $1 < S_x$ for all $\kappa_x \in \Omega$. So $S_x = ZP$ and hence $zx \sim_Q x$ whenever $x \in P-ZP$.\\

Assume that the conditions hold, and observe that they imply $J'I = IJ'$ since $C = ZP \lhd P$. Furthermore, $I \simeq k \times \cdots \times k$ as a $kC$-module since $S_x = C$ for $x \in P-C$, and hence $J'^dI^{p-d} = 0$ for $1 \leq d \leq p$. It remains to show that $I^p = 0$, and for this it suffices to show that $I^2 = 0$. By Lemma 3.3, if $\kappa_x,\kappa_y,\kappa_z \in \Omega$ then $\kappa_z$ occurs in $\kappa_x\kappa_y$ with nonzero coefficient $\mu_{xyz}$ only when $z \sim_Q q^{-1}xqy$ for some $q \in Q$, in which case $\mu_{xyz} = |q^{-1}S_{x^{-1}}^{\dagger}q \cap S_y^{\dagger}|$. For $c \in ZP$ and $w \in P-ZP$ there is $q_1 \in Q$ for which $c = [q_1,w]$. Thus, for all $q_2 \in Q$ we obtain $c[q_2,w] = [q_2q_1,w] \in S_w^{\dagger}$. In other words, left multiplication by $ZP$ leaves $S_w^{\dagger}$ invariant. In particular, $ZP$ acts semiregularly on $q^{-1}S_{x^{-1}}^{\dagger}q \cap S_y^{\dagger}$, and so $\mu_{xyz} = 0$. Thus, $I^2 = 0$ and $\ell\ell(kP^Q) = p$.
\end{proof}

\begin{remark}Theorem 4.1 implies the existence of infinitely many non-isomorphic $p$-groups $P$ for which $\ell\ell(ZkP) = p$. This stands in stark contrast with the case of groups algebras. More precisely, since $\ell\ell(kP) \geq n(p-1)+1$ whenever $|P| = p^n$, there are only finitely many groups satisfying $\ell\ell(kP) = d$ for any fixed $d$. Observe also, that if $P$ and $Q$ are $p$-groups for which $\ell\ell(ZkP) = \ell\ell(ZkQ) = p$, then by Proposition 8.2 we have $ZP \simeq ZQ$ and $\ell\ell(k(P * Q)) = p$ where $*$ denotes the central product of $P$ and $Q$.
\end{remark}

\section{Upper Bounds on Loewy Length}

Using the computations from section 7 we can derive precise upper bounds on $\ell\ell(kP^Q)$, in analogy with work done in ~\cite{Koshitani}, ~\cite{Motose}, and ~\cite{Motose2}.

\begin{thm}Let $P$ be a $p$-group of order $p^n$ with subgroup $Q$ and $k$ a field of characteristic $p$. Then the following hold.\\

\begin{enumerate}

\item If $Q \leq R \leq P$ then $\ell\ell(kP^R) \leq \ell\ell(kP^Q)$.\\

\item If $Q$ is noncentral then $\ell\ell(kP^Q) < \ell\ell(kP)$.\\

\item Either $\ell\ell(kP^Q) < p^{n-1}$ or $\ell\ell(kP^Q) \in \{p^{n-1},p^{n-1}+p-1,p^n\}$.\\

\item $\ell\ell(kP^Q) = p^n$ if and only if $P \simeq \mathbb{Z}_{p^n}$.\\

\item $\ell\ell(kP^Q) = p^{n-1}+p-1$ if and only if $P$ is noncyclic with a cyclic subgroup of index $p$ and $Q \leq ZP$.\\

\item $\ell\ell(kP^Q) = p^{n-1}$ if and only if $P \simeq \mathbb{Z}_2 \times \mathbb{Z}_2 \times \mathbb{Z}_2$; $P$ is the extra special group of 27 with exponent 3 and $Q \leq ZP$; or $P$ is noncyclic and $C_P(Q)$ is a cyclic subgroup of index $p$ in $P$.

\end{enumerate}

\end{thm}

\begin{proof}
If $Q \leq R$ then $kP^R \subseteq kP^Q$ and so $J(kP^R) = J(kP^Q) \cap kP^R \subseteq J(kP^Q)$, thus establishing (1). Suppose that $\ell\ell(kP^Q) = \ell\ell(kP)$, and write $d+1 = \ell\ell(kP)$. Since $J^i(kP) = \mbox{Soc}^{d+1-i}(kP)$ for all $i$, we see that $J^d(kP)$ is 1-dimensional and contains $\sigma = \sum_{x \in P} x$. Then $\sigma \in J^d(kP^Q)$ since $0 \not= J^d(kP^Q) \subseteq J^d(kP)$. Using $J(kP^Q) = J(kC_P(Q)) \oplus I$, we can write $\sigma = \sigma_1 + \sigma_2$ where $\sigma_1 \in J^d(kC_P(Q))$ and $\mbox{Supp}(\sigma_2) \subseteq P-C_P(Q)$. Since $\mbox{Supp}(\sigma_1) \subseteq C_P(Q)$, we see that $\sigma_1 = \sum_{x \in C_P(Q)} x$. But since $J(kC_P(Q)) = kC_P(Q) \cap J(kP)$ we obtain $J^d(kC_P(Q)) \subseteq J^d(kP)$ and hence $\sigma_1 = \sigma$. This implies that $C_P(Q) = P$, and thus $Q \leq ZP$, contrary to our assumption on $Q$. We have thus established (2).\\

By (1) we know that $\ell\ell(kP^Q) \leq \ell\ell(kP)$ and $\ell\ell(kP) \leq p^n$ by application of Nakayama's Lemma. Moreover, if $\ell\ell(kP^Q) = p^n$ then $\ell\ell(kP) = p^n$, and by ~\cite{Motose} this implies that $P \simeq \mathbb{Z}_{p^n}$. The converse is obvious, thus establishing (4).\\

If $\ell\ell(kP^Q) < p^n$ then $P$ is not cyclic and hence $\ell\ell(kP^Q) \leq \ell\ell(kP) \leq p^{n-1}+p-1$ by ~\cite{Koshitani}. So if $\ell\ell(kP^Q) = p^{n-1}+p-1$ then $\ell\ell(kP) = p^{n-1}+p-1$, and hence $P$ contains an element of order $p^{n-1}$. Corollary 7.4 implies that $Q \leq ZP$; thus establishing (5).\\

If $\ell\ell(kP^Q) = p^{n-1}$ then $P$ is not cyclic and $\ell\ell(kP) \geq p^{n-1}$. By ~\cite{Koshitani} and ~\cite{Motose2} either $\ell\ell(kP) = p^{n-1}$ or $\ell\ell(kP) = p^{n-1}+p-1$. If $\ell\ell(kP) = p^{n-1}$ then $P \simeq \mathbb{Z}_2 \times \mathbb{Z}_2 \times \mathbb{Z}_2$ or $P$ is the extra special group of order 27 with exponent 3. In the latter case, if the Loewy series for $kP^Q$ has Poincar\'{e} polynomial $p(t)$, then $\ell\ell(kP^Q) = \deg p(t)+1$. In particular, Theorem 4.1 implies that $\ell\ell(kP^Q) \leq 2\log_3 |C_P(Q)|+3 \leq 9$ with equality precisely when $C_P(Q) = P$ and hence $Q \leq ZP$. If $\ell\ell(kP) = p^{n-1}+p-1$ then Corollary 7.4 yields $\ell\ell(kP^Q) = p^{n-1}$ precisely when $C_P(Q)$ is cyclic of index $p$ in $P$. This establishes (6) and completes (3).
\end{proof}

\begin{remark}It is natural to ask whether (2) of Theorem 9.1 generalizes in the following way: if $P$ is a $p$-group with subgroups $Q \leq R$ and $kP^R \subsetneq kP^Q$, must it follow that $\ell\ell(kP^R) < \ell\ell(kP^Q)$? This is not the case: take $P = D_{16}, Q = \lr{a^2}$, and $R = \lr{a}$, so that $\dim kP^R = 10, \dim kP^Q = 12$, and $\ell\ell(kP^R) = 8 = \ell\ell(kP^Q)$ by Theorem 7.3 and the remarks following Corollary 7.4.
\end{remark}

\section{Representation Type of $kP^Q$}

It is known that $kP$ has only finitely many non-isomorphic indecomposable modules precisely when $P$ is cyclic. Here we establish an analogous result for $kP^Q$. As before, write $J' = J(kC_P(Q))$ and $J = J(kP^Q)$. It is necessary to assume in this section that $k$ is algebraically closed. We first need the following lemma.

\begin{lemma}
\label{structure}
If $Q$ is a non-central subgroup of the $p$-group $P$ and we write $kP^Q = kC \ltimes I$, then $I \not\subseteq J^2$.
\end{lemma}
\begin{proof}
Since $J = J' \oplus I$ we have $J^2 = J'^2 \oplus (J'I + IJ' + I^2)$. If $I \subseteq J^2$ then $I = (J'+I)I + IJ' = JI + IJ'$. Considering $I$ as a left $kP^Q$-module, we see $I = IJ'$ by Nakayama's Lemma. Considering $I$ as a right $kC$-module, Nakayama's Lemma implies that $I = 0$. Thus $Q \leq ZP$, contrary to our assumption.
\end{proof}

\begin{thm}If $P$ is a $p$-group with subgroup $Q$ and $k$ is an algebraically closed field of characteristic $p$, then $kP^Q$ has finite representation type precisely when $P$ is cyclic.\end{thm}
\begin{proof}
If $P$ is cyclic, then $kP^Q = kP$ has finite representation type. If $P$ is not cyclic and $Q \leq ZP$, then $kP^Q = kP$ has infinite representation type. So assume that $P$ is not cyclic and $Q$ is non-central. From $kP^Q = kC \ltimes I$ and the fact that $I$ is nilpotent we get

$$J'^2 \simeq J^2(kP^Q/I) = (J^2+I)/I \simeq J^2/J^2 \cap I$$

Using this result and counting dimensions yields

\begin{equation}\nonumber
\begin{split}
\dim kP^Q - \dim J^2 &= \dim kC + \dim I - \dim J'^2 - \dim J^2 \cap I \\
&= \dim kC - \dim J'^2 + \dim I - \dim J^2 \cap I \\
&\geq 2 + \dim I - \dim J^2 \cap I \\
&> 2
\end{split}
\end{equation}\nonumber

where the first inequality follows from the fact that $J'^2 \subsetneq J' \subsetneq kC_P(Q)$ since $|C_P(Q)| \ge 2$, and the second inequality follows from the assertion in Lemma 10.1 that $I \not\subseteq J^2$. Therefore, $\dim J^2 < \dim kP^Q-2$ and in particular $\dim J/J^2 \ge 2$.\\

Now $kP^Q$ is a basic algebra that is split over $k$ since $kP^Q/J \simeq k$. So we assign to $kP^Q$ its ordinary quiver $\mathcal{Q}$ as in ~\cite{Elements}. Hence $\mathcal{Q}$ is the directed graph with a single vertex $e_0$ (corresponding to the primitive idempotent $1$) with loops $\alpha : e_0 \rightarrow e_0$ indexed by a basis $\{x_{\alpha}\}$ of $J/J^2$. If $k\mathcal{Q}$ denotes the path algebra associated with $\mathcal{Q}$, then the map $\mathcal{Q} \rightarrow kP^Q$ given by $e_0 \mapsto 1$ and $\alpha \mapsto x_{\alpha}$ extends to an algebra homomorphism $k\mathcal{Q} \rightarrow kP^Q$. Moreover, this map is surjective with kernel contained in $R^2$ where $R$ is the arrow ideal of $k\mathcal{Q}$ generated by $\{\alpha\}$. Therefore $kP^Q/J^2 \simeq k\mathcal{Q}/R^2 \simeq k[X_1,\ldots,X_n]/(X_1,\ldots,X_n)^2$ where $n = \dim J/J^2 \ge 2$. In particular, the three dimensional algebra $\Lambda = k[x,y]/(x,y)^2$ is a homomorphic image of $kP^Q$. It is shown in ~\cite{Auslander} that $\Lambda$ has infinitely many non-isomorphic indecomposable modules. The same must therefore be true for $kP^Q$, thus establishing the result.
\end{proof}

Unfortunately, it is not clear how to modify this argument to settle the more subtle question of when $kP^Q$ has wild type and when it has tame type.

\section{Open Questions}

In computing $J^d(kP^Q)$ we often took for granted condition (*), and used condition (**) whenever it held. As remarked after Corollary 3.2, condition (*) appears to be quite mild. It would be nice to have a justification of this observation beyond Corollary 3.2. For instance, computations suggest that it holds whenever $|Q|^2 > |P|$ or $|C_P(Q)|^2 < |P|$. On the other hand, condition (**) does not appear to hold in most cases. Thankfully, we were able to work around this obstacle as in section 7.  Corollary 3.5 provides a criterion for detecting when (**) fails. Interestingly, of the 7,347 many pairs $(P,Q)$ with $|P| = 2^6$ where (**) fails, precisely 5,588 many of them satisfy the hypotheses of Corollary 3.5.\\

On a related note, it is natural to try to generalize the result from section 5 on the symmetry of the Loewy series of $kP^Q$ without the restriction that (**) holds. As a clue, it appears that if symmetry holds, then $\ell\ell(kP^Q)-\ell\ell(kC_P(Q)) = 0,1$. This is easy to prove if $|C_P(Q)| = 2$ for instance. At the very least, it would be interesting to verify that symmetry can only occur when $p = 2, 3$.\\

In a different direction, a more detailed analysis in Proposition 6.4 shows that if $\ell\ell(kP^Q) \not= \ell\ell(kC_P(Q))$ then $\ell\ell(kP^Q) \leq |\Omega|$. It is suggested by computational evidence that if $\ell\ell(kP^Q) \not= \ell\ell(kC_P(Q))$ then in fact $\ell\ell(kC_P(Q)) \leq |\Omega|-1$, with equality only for $p = 2$ and the examples mentioned after Theorem 7.3. Such a result would provide an alternative form of the upper bounds provided in section 9.

\section*{Acknowledgements}

This paper is part of the author's PhD thesis at the University of Chicago. The author thanks Jon Alperin for his helpful supervision. The author also thanks the authors of MAGMA ~\cite{MAGMA}. This system was quite helpful in gaining insight into Loewy structure of $kP^Q$ as well as leading to the formulation of a number of our results.

\bibliographystyle{plain}
\bibliography{AAAbibfile}

\begin{thebibliography}{10}

\bibitem{AlperinLoewy}
J.~L. Alperin.
\newblock Loewy structure of permutation modules for {$p$}-groups.
\newblock {\em Quart. J. Math. Oxford Ser. (2)}, 39(154):129--133, 1988.

\bibitem{Aschbacher}
M.~Aschbacher.
\newblock {\em Finite group theory}, volume~10 of {\em Cambridge Studies in
  Advanced Mathematics}.
\newblock Cambridge University Press, Cambridge, second edition, 2000.

\bibitem{Elements}
Ibrahim Assem, Daniel Simson, and Andrzej Skowro{\'n}ski.
\newblock {\em Elements of the representation theory of associative algebras.
  {V}ol. 1}, volume~65 of {\em London Mathematical Society Student Texts}.
\newblock Cambridge University Press, Cambridge, 2006.
\newblock Techniques of representation theory.

\bibitem{Auslander}
Maurice Auslander, Idun Reiten, and Smal{\o}Sverre O.
\newblock {\em Representation theory of {A}rtin algebras}, volume~36 of {\em
  Cambridge Studies in Advanced Mathematics}.
\newblock Cambridge University Press, Cambridge, 1997.
\newblock Corrected reprint of the 1995 original.

\bibitem{BensonI}
David~J. Benson.
\newblock {\em Representations and cohomology. {I}}, volume~30 of {\em
  Cambridge Studies in Advanced Mathematics}.
\newblock Cambridge University Press, Cambridge, second edition, 1998.
\newblock Basic representation theory of finite groups and associative
  algebras.

\bibitem{MAGMA}
Wieb Bosma, John Cannon, and Catherine Playoust.
\newblock The {M}agma algebra system. {I}. {T}he user language.
\newblock {\em J. Symbolic Comput.}, 24(3-4):235--265, 1997.
\newblock Computational algebra and number theory (London, 1993).

\bibitem{CurtisReiner}
Charles~W. Curtis and Irving Reiner.
\newblock {\em Representation theory of finite groups and associative
  algebras}.
\newblock AMS Chelsea Publishing, Providence, RI, 2006.
\newblock Reprint of the 1962 original.

\bibitem{Ellers94}
Harald Ellers.
\newblock Cliques of irreducible representations of {$p$}-solvable groups and a
  theorem analogous to {A}lperin's conjecture.
\newblock {\em Math. Z.}, 217(4):607--634, 1994.

\bibitem{Ellers95}
Harald Ellers.
\newblock The defect groups of a clique, {$p$}-solvable groups, and {A}lperin's
  conjecture.
\newblock {\em J. Reine Angew. Math.}, 468:1--48, 1995.

\bibitem{Ellers00}
Harald Ellers.
\newblock Searching for more general weight conjectures, using the symmetric
  group as an example.
\newblock {\em J. Algebra}, 225(2):602--629, 2000.

\bibitem{Ellers04}
Harald Ellers and John Murray.
\newblock Block theory, branching rules, and centralizer algebras.
\newblock {\em J. Algebra}, 276(1):236--258, 2004.

\bibitem{Fine}
N.~J. Fine.
\newblock Binomial coefficients modulo a prime.
\newblock {\em Amer. Math. Monthly}, 54:589--592, 1947.

\bibitem{IsaacsGroups}
I.~Martin Isaacs.
\newblock {\em Finite group theory}, volume~92 of {\em Graduate Studies in
  Mathematics}.
\newblock American Mathematical Society, Providence, RI, 2008.

\bibitem{Jennings}
S.~A. Jennings.
\newblock The structure of the group ring of a {$p$}-group over a modular
  field.
\newblock {\em Trans. Amer. Math. Soc.}, 50:175--185, 1941.

\bibitem{Koshitani}
Shigeo Koshitani.
\newblock On the nilpotency indices of the radicals of group algebras of
  {$P$}-groups which have cyclic subgroups of index {$P$}.
\newblock {\em Tsukuba J. Math.}, 1:137--148, 1977.

\bibitem{Motose2}
Kaoru Motose.
\newblock On a theorem of {S}. {K}oshitani.
\newblock {\em Math. J. Okayama Univ.}, 20(1):59--65, 1978.

\bibitem{Motose}
Kaoru Motose and Yasushi Ninomiya.
\newblock On the nilpotency index of the radical of a group algebra.
\newblock {\em Hokkaido Math. J.}, 4(2):261--264, 1975.

\bibitem{Murray02}
John Murray.
\newblock Squares in the centre of the group algebra of a symmetric group.
\newblock {\em Bull. London Math. Soc.}, 34(2):155--164, 2002.

\bibitem{Murray04}
John Murray.
\newblock Generators for the centre of the group algebra of a symmetric group.
\newblock {\em J. Algebra}, 271(2):725--748, 2004.

\bibitem{Towers}
Matthew Towers.
\newblock Endomorphism algebras of transitive permutation modules for
  {$p$}-groups.
\newblock {\em Arch. Math. (Basel)}, 92(3):215--227, 2009.

\bibitem{Wallace}
D.~A.~R. Wallace.
\newblock Lower bounds for the radical of the group algebra of a finite
  {$p$}-soluble group.
\newblock {\em Proc. Edinburgh Math. Soc. (2)}, 16:127--134, 1968/1969.

\end{thebibliography}

\end{document}